\documentclass[12pt]{article}
\usepackage{authblk}
\usepackage{graphics}
\usepackage{graphicx}
\usepackage[mathscr]{eucal}
\usepackage{epic,eepicemu,eepic}
\usepackage{amsmath,amsxtra,amssymb,latexsym,amscd,amsthm,amsfonts}
\usepackage[left=3cm,right=1.5cm,top=2.5cm,bottom=2.5cm]{geometry}
\newtheorem{thm}{Theorem}

\newtheorem{lem}[thm]{Lemma}

\theoremstyle{definition}
\newtheorem{rem}[thm]{Remark}

\DeclareMathOperator{\divv}{div}
\DeclareMathOperator{\Vol}{Vol}

\DeclareMathOperator{\trace}{trace}

\begin{document}
\baselineskip=17pt
\title{\bf On entire $f$-maximal graphs in the Lorentzian product $\Bbb G^n\times\Bbb R_1$}
 \small{}
\author[a]{H. V. Q. An}
\author[a]{D. V. Cuong} 
\author [b]{N. T. M. Duyen}
\author[b*]{ D. T. Hieu}
\author[c]{T. L. Nam}
\affil[a]{Duy Tan University, Danang, Vietnam}
\affil[b]{College of Education, Hue University, Hue, Vietnam}
\affil[c]{Dong Thap University, Dong Thap, Vietnam}

 \maketitle
\begin{abstract}
In the Lorentzian product $\Bbb G^n\times\Bbb R_1,$ we give a comparison theorem between the $f$-volume of an entire $f$-maximal graph  and  the $f$-volume  of the hyperbolic $H_r^+$ under the condition that the gradient of the function defining the graph is bounded away from 1. This condition  comes from an example of non-planar entire $f$-maximal graph in $\Bbb G^n\times\Bbb R_1$ and is equivalent to  the hyperbolic angle function of the graph being bounded. As a consequence, we obtain a Calabi-Bernstein type theorem for $f$-maximal graphs in $\Bbb G^n\times\Bbb R_1.$  
 \end{abstract}
 \vskip 1cm

%====================================
\section{Introduction}
Suppose  that $\Sigma$ is the graph of a $C^2$ function $f$
	over a domain $U\subset\Bbb R^n.$
	The graph $\Sigma$ is minimal if the function $f$ satisfies the Minimal Surface Equation
	$$\divv\left(\frac{\nabla f}{\sqrt{1+|\nabla f|^2}}\right)=0.$$
	S. N. Bernstein (1915-1917) proved a surprising  theorem in the case $n=2$, named after him, that  such a  function defined over $\Bbb R^2$ must be affine. This means that a minimal graph over $\Bbb R^2$ must be a plane. He conjectured that the theorem holds true for $n>2.$
	
Bernstein's conjecture has been a longstanding problem and it was proved to be true for $n \le 7$ by the works of  De Giorgi, F. Almgren  and J. Simons (see \cite{alm}, \cite{gio} and  \cite{si}). However, Bombieri, De Giorgi  and  Giusti \cite{bomgiogiu} 
showed the existence of some entire minimal graphs  other than the hyperplanes for $n\ge 8.$  

The Lorentzian version of the Bernstein's Theorem for maximal graphs in the Lorentz-Minkowski spaces $\Bbb R^{n+1}_1$ is called the Calabi-Bernstein's Theorem. Quite different from the Euclidean case, the Calabi-Bernstein's Theorem holds true for any dimension. The theorem was first proved by Calabi \cite{ca} for the case $n=2$ and later by Cheng and Yau \cite{chengyau} for the case $n>2.$ 

The Calabi-Bernstein's Theorem has been generalized to  Lorentzian product spaces (see \cite{alal1}, \cite{alal2} and  \cite{alroru1}, for instance) as well as to other ambient spaces such as warped products, Robertson-Walker spacetimes, manifolds with density\ldots (see \cite{alroru2}, \cite{calisa}, \cite{hina}, \cite{ro} and  \cite{wan},  for instance).  
 Albujer and Al\'ias \cite{alal2} proved a Calabi-Bernstein type result, that any entire maximal graph in $M^2\times\Bbb R_1,$ where $M$ is a Riemannian manifold with non-negative Gaussian curvature,  must be totally geodesic. They also proved that if $M$ is  non-flat, then the graph must be a slice $M\times\{t_0\}, t_0\in\Bbb  R.$ In \cite{al}, some non-planar entire maximal graphs in $\Bbb H^2\times \Bbb R_1$ were constructed. This shows that without the assumption of nonnegative Gauss curvature of $M,$ the theorem is no longer true. 

The Calabi-Bernstein's Theorem has been also generalized to manifolds with density. 
L. Wang (see \cite{wan}) proved a Bernstein type theorem for self-shrinkers  in $\Bbb R^n,$ i.e., for $f$-minimal hypersurfaces in Gauss space $\Bbb G^n.$ 
By a calibration argument, without using second order differential equations, the last two authors of this paper (see \cite{hina}) proved a Bernstein type theorem for entire $f$-minimal graphs  in $\Bbb G^n\times \Bbb R.$ In these cases the theorem holds true without any condition.

The purpose of this paper is to  establish a similar Calabi-Bernstein type theorem for  entire $f$-maximal graphs in the Lorentzian product $\Bbb G^n\times\Bbb R_1$ by using Lorentzian calibration arguments.

The results  are quite different from the case of the Riemannian product.  The first main result of the paper is  that an example of  a non-planar entire $f$-maximal graph  in $\Bbb G^n\times \Bbb R_1$  was found.  We see that along the first coordinate axis, the gradient of the function determining the graph converges to 1. So it seems  the condition  for an entire $f$-maximal graph in $\Bbb G^n\times\Bbb R_1$ being a hyperplane  is that the gradient of the function determining the graph is  bounded away from 1, or equivalently, the hyperbolic angle function of the graph is bounded.  In fact, with this condition,  a comparison theorem between the $f$-volume of an entire $f$-maximal graph  and  the $f$-volume  of the hyperbolobic $H_r^+$ is given.  As a consequence,  we obtain the second main result, a Calabi-Bernstein type theorem for $f$-maximal graphs in $\Bbb G^n\times \Bbb R_1.$ 

For more details about Lorentz-Minkowski spaces, maximal surface, calibrations, manifolds with density and Gauss space, we refer the reader to \cite{hala}, \cite{hi}, \cite {hina},  \cite{lo},  \cite{me}, \cite {mo1}, \cite{mo2} and \cite{nei}.

\section{$f$-maximal  graphs}
%\subsection {$f$-maximal hypersurfaces} 
Let $\Bbb R^{n+1}_1$ be the $(n+1)$-dimensional Lorentz-Minkowski space, i.e., $\Bbb R^{n+1}$ endowed with the Lorentzian scalar product
$$\langle \ , \ \rangle=dx_1^2+dx_2^2+\ldots +dx_n^2-dx_{n+1}^2.$$
Since $\langle, \rangle$ is
non-positive definite, for $\textbf x\in\Bbb R^{n+1}_1,$ \ $\langle \textbf x, \textbf x\rangle$ may be zero or negative. A nonzero
vector $\textbf x\in \Bbb R^{n+1}_1$ is called spacelike, lightlike or timelike if  $\langle \textbf x, \textbf
x\rangle>0$, $\langle \textbf x, \textbf x\rangle=0$ or $\langle \textbf x, \textbf x\rangle<0$,
respectively. Two vectors $\textbf x, \textbf y\in \Bbb R^{n+1}_1$  are said to be Lorentzian orthogonal if $\langle \textbf x, \textbf y\rangle=0.$  The norm of a tangent vector $\textbf x$ is defined by $\|\textbf x\|=\sqrt{|\langle\textbf x, \textbf x\rangle|}.$

 We say that two timelike tangent vectors $\textbf x$ and $\textbf y$  lie in the same timelike cone if $\langle \textbf x, \textbf y\rangle < 0.$  If  $\textbf x$ and $\textbf e_{n+1},$ where $\textbf e_{n+1}$ denotes the timelike coordinate vector field, lie in the same timelike cone, we say that $\textbf x$ is future-directed. For timelike tangent vectors $\textbf x$ and $\textbf y,$ we have the Cauchy-Schwarz inequality
$$|\langle \textbf x, \textbf y\rangle| \ge\sqrt{-\langle \textbf x, \textbf x\rangle}\sqrt{-\langle \textbf y, \textbf y\rangle},$$
and the equality holds if and only if $\textbf x$ and $\textbf y$ are proportional. If $\textbf x$ and $\textbf y$ lie in the same timelike cone, then there exists a unique number $\varphi \ge 0$ such that
$$\langle \textbf x, \textbf y\rangle=-|\textbf x||\textbf y| \cosh \varphi.$$
The number $\varphi$  is called the hyperbolic angle between $\textbf x$ and $\textbf y.$ 
%================================================

A hypersurface  in  $\Bbb R^{n+1}_1$ is called spacelike if its induced metric
from $\Bbb R^{n+1}_1$ is Riemannian. Equivalently,  every tangent vector of the hypersurface is spacelike.  A spacelike hypersurface is orientable, i.e., there exists a unit timelike normal vector field $N$ on it.  Moreover, we can suppose that $N$ is future-directed and then we can define the hyperbolic angle function $\theta$ of the hypersurface. The value of $\theta$ at a point $p$ on the hypersurface is the hyperbolic angle between $N(p)$ and $\textbf{e}_{n+1}.$

Let $\Sigma$ be a spacelike hypersurface in $\Bbb R^{n+1}_1$ and $N$ be the unit timelike normal vector field. 
 Denote by $\overline{D}$ and $D$ the Levi-Civita connections of $\Bbb R^{n+1}_1$ and $\Sigma,$ respectively. For $X,Y\in \mathfrak{X}(\Sigma),$  we have
$$\overline{D}_XY=D_XY-\langle A(X), Y\rangle N,$$
and 
$$ A(X)=-\overline{D}_XN,$$
where $A:\mathfrak{X}(\Sigma)\rightarrow \mathfrak{X}(\Sigma)$ is the shape operator associated to $N.$
The mean curvature of $\Sigma,$ denoted by $H,$  is defined by
$$H=-\frac 1n\trace(A).$$
It is well-known that
$$nH=-\divv N.$$
A spacelike hypersurface is called \emph{maximal} if its mean curvature $H$ is zero everywhere.  
 When  $\Sigma$ is the graph of a function $u:U\rightarrow\Bbb R,$  where $U\subset \Bbb R^n,$ then it is easy to see that 
$$N=\frac 1{\sqrt{1-|\nabla u|^2}}(\nabla u+\textbf e_{n+1}),$$
is the future-directed unit timelike normal vector field on $\Sigma,$ where $\nabla u$ denotes the gradient of the function $u$ on $U.$ Note that since $\Sigma$ is spacelike, $|\nabla u|<1.$ Hence the mean curvature of $\Sigma$ is  computed on $\Bbb R^n$ to be
$$H=-\frac 1n{\rm div}_{_{\Bbb R^n}}\left(\frac{\nabla u}{\sqrt{1-|\nabla u|^2}}\right),$$
where ${\rm div}_{_{\Bbb R^n}}$ is the divergence operator on $\Bbb R^n.$

Therefore, $\Sigma$ is maximal if and only if $u$ satisfies the maximal equation
$${\rm div}_{_{\Bbb R^n}}\left(\frac{\nabla u}{\sqrt{1-|\nabla u|^2}}\right)=0.$$

A density on $\Bbb R^{n+1}_1$ is a positive function $e^{-f},$ and it is used to weight volumes of $k$-dimensional spacelike submanifolds. In $\Bbb R^{n+1}_1$ with density $e^{-f},$ the weighted volume, or $f$-volume, of a spacelike hypersurface $\Sigma,$  denoted by $\Vol_f(\Sigma),$  is defined by
$$\Vol_f(\Sigma)=\int_{\Sigma}e^{-f}dV_\Sigma,$$
where $dV_\Sigma$ is the volume element of the hypersurface. The weighted mean curvature or $f$-mean curvature of $\Sigma,$ denoted by $H_f,$ is defined by 
$$H_f=H+\frac 1n\langle  \nabla f, N\rangle.$$
 Now, $\Sigma$ is called \emph{$f$-maximal} provided that  $H_f=0$ everywhere, i.e., $H=-\frac 1n\langle \nabla f, N\rangle.$ 

If  $\Sigma$ is the graph of a function $u:U\rightarrow\Bbb R,$ then $\Sigma$ is $f$-maximal if and only if $u$ satisfies the $f$-maximal equation
$${\rm div}_{_{\Bbb R^n}}\left(\frac{\nabla u}{\sqrt{1-|\nabla u|^2}}\right)-\langle \nabla f, N\rangle=0.$$
The Gauss space $\Bbb G^n,$ a typical example of  manifolds with density, is  just $\Bbb R^n$ with Gaussian probability density
$$e^{-f(\textbf x)}=e^{c-\frac{|{\bf x}|^2}2},$$
where $c=\log [(2\pi)^{-\frac n2}].$  Gauss space has many applications to probability and statistics (see \cite{mo1} and  \cite{mo2}, for instance).
  
The space $\Bbb G^n\times\Bbb R_1$ is $\Bbb R^{n+1}_1=\Bbb R^n\times\Bbb R_1$ endowed with the Gaussian-Euclidean density
$$e^{-f(\textbf x, t)}=e^{c-\frac{|{\bf x}|^2}2},$$
where $ (\textbf x, t)\in\Bbb G^n\times \Bbb R_1, \ \textbf x\in \Bbb G^n,\ t\in \Bbb R_1.$

Below is the example about a non-planar entire $f$-maximal graph in $\Bbb G^n\times \Bbb R_1$  mentioned in the introduction section.

Consider the entire graph $\Sigma_0$  of the function $u:\Bbb{G}^n\longrightarrow\mathbb{R}, \  \textbf x\longmapsto\displaystyle \int_0^{x_1} \sqrt{\dfrac{e^{\tau^2}}{1+e^{\tau^2}}}d\tau,$ where $\textbf x=(x_1, x_2,\ldots, x_n).$ The graph $\Sigma_0$ has a parametrization as follows
\begin{align}\nonumber
X(\textbf x,t)=\left(\textbf x, u(\textbf x)\right).
\end{align}
Since
$$ 1-|\nabla u|^2=\frac 1{1+e^{x_1^2}}>0,$$
 $\Sigma_0$ is spacelike. 

It follows that the future-directed unit timelike normal vector field is
$${N}=\left( \dfrac{u_{x_1}}{\sqrt{1-u_{x_1}^2}},0,\ldots,0, \dfrac{1}{\sqrt{1-u_{x_1}^2}}\right).$$
Therefore, 
the mean curvature of $\Sigma_0$ is
 $$H=-\frac 1n{\rm div} N=-\frac 1n\frac{u_{x_1x_1}}{(1-u_{x_1}^2)^{3/2}}=-\frac{x_1e^{x_1^2/2}}n.$$
Since 
$$\langle\nabla f, N\rangle=\frac{x_1u_{x_1}}{\sqrt{1-u_{x_1}^2}} =x_1e^{x_1^2/2},$$
$$H_f=H+\frac 1n\langle\nabla f, N\rangle=0,$$
i.e., $\Sigma_0$ is $f$-maximal. 

This example shows that we need an assumption in order to obtain a uniqueness result.
%=========================================================
\section{Calabi-Bernstein result}
 Let $\Sigma \subset \Bbb G^n\times\Bbb R_1$ be  the graph of a function $u({\textbf x})=t$ over $\Bbb G^n,$ that is $f$-maximal, and 
 $N$ be the  future-directed unit timelike normal vector field of $\Sigma.$ Consider the smooth extension  of $N$ by translations along $t$-axis, that we will also denote by $N,$ and the $n$-differential form $w$ defined by
$$w=i_NdV.$$
 Similarly as in the Riemannian case, we obtain the following,
\begin{lem} \label{lemw}
The form $w$ is a Lorentzian calibration that calibrates $\Sigma,$ i.e.,
 \begin{enumerate}
 \item $d(e^{-f}w)=0;$ 
\item  For orthonormal spacelike vector fields  $X_i,\ i=1,2,\ldots, n,$ \ $|w(X_1, X_2,\ldots, X_{n})|\ge 1,$  and the equality holds at a point if and only if  $X_1, X_2,\ldots, X_{n}$ are tangent to $\Sigma$ at such a point. 
 \end{enumerate}
\end{lem}
\begin{proof}
 \begin{enumerate}
\item  Because $\Sigma$ is $f$-maximal,

%============================
\begin{align}\nonumber
d(e^{-f} w)&= d(e^{-f} i_NdV)={\rm div}(e^{-f}{N})dV\\
 &= (e^{-f}{\rm div} N-e^{-f}\langle\nabla f, N\rangle)dV\nonumber\\
 &=-e^{-f}(nH+\langle\nabla f, N\rangle)dV=0.\nonumber
\end{align}

%======================
\item At any point $p$ of the Lorentz Minkowski spacetime, we can find a tangent vector $\overline{N}(p)$ such that $\{X_1(p),\ldots, X_n(p), \overline{N}(p)\}$ is a frame \cite[page 84]{nei}. Then we can write
$$N=\pm\cosh \varphi\overline{N}+\sinh \varphi u_0,$$
for $\varphi\in\Bbb R,$ where $\langle N, u_0\rangle=0,\ \langle u_0, u_0\rangle=1,$ and the sign depends on the timelike orientation. Hence, 
\begin{align*}
|w(X_1, \ldots, X_n)(p)|=|dV(X_1,\ldots, X_n, N)(p)|=|dV(X_1,\ldots, X_n, \pm\cosh\varphi\overline{N})|=\cosh\varphi.
\end{align*}
The previous equation clearly proves this part of the lemma.
 \end{enumerate}
\end{proof}
 By Stokes' Theorem, as in the case of $f$-minimal submanifolds \cite [Theorem 2.1]{hi},  it can be proved that $\Sigma$ is $f$-area-maximizing, i.e., any compact portion of $\Sigma$ has largest area among all spacelike hypersurfaces in its homology class, the Lorentzian version follows analogously.

As a next step, we establish a  comparison  between  the $f$-volume of $\Sigma$ and the $f$-volume of the spacelike Lorentzian hypersphere
 $H_r^+=\{({\textbf x},t)\in\Bbb G^{n}\times\Bbb R_1 :\ \langle \textbf x,\textbf x\rangle-\langle t, t\rangle=-r^2,\ t>0\}.$  

\begin{lem} \label{lemvol}
If $|\nabla u|$ is bounded away form 1, then for every $r>0$,
$$\Vol_f(\Sigma)\ge \Vol_f(H^+_r).$$
\end{lem}
%=========================================
\begin{rem} Because $\langle N, \textbf {e}_{n+1}\rangle =\displaystyle\frac 1{\sqrt {1-|\nabla u|^2}},$  it follows that $|\nabla u|$ is bounded from 1 is equivalent to the hyperbolic angle function $\theta$ of $\Sigma$ being bounded. This condition appeared in some Calabi-Bernstein type results for constant mean curvature spacelike hypersurfaces in Lorentzian warped products (see \cite {alroru2}, for instance).
\end{rem}
%==============================
\begin{proof}
\begin{figure}[h]
	\centering
	\includegraphics{Lemma-1}
	\caption{\textit{For the proof of Lemma 2}}
	\label{fig:lemma}
\end{figure}
Since the density does not depend on the last coordinate, the $f$-volume and $f$-mean curvature are invariant under translations along the $x_{n+1}$-axis.  Therefore, we can assume that the origin $O\in \Sigma.$ 
Let ${B^n_R}$ be the $n$-ball in $\Bbb G^n$ with center $O$ and radius $R$ and ${\cal C} =B^n_R\times \Bbb R$  be the $(n+1)$-cylinder over ${B^n_R}.$
 Since $\Sigma$ is spacelike, $\widetilde{\Sigma}:=\Sigma\cap {\cal C}$ is bounded and lies in the slab between parallel hyperplanes $t=\pm R.$ Let $\Omega$ be the region bounded by ${\cal C}, \widetilde{\Sigma}$ and $ \widetilde{H^+_r}:=H^+_r\cap {\cal C}$  with inward-pointing orientation; $\widetilde{\partial\cal C}$ be the part of $\partial{\cal C}$ between  $\widetilde{\Sigma}$ and $\widetilde{H^+_r} $ and $w$ be the Lorentzian calibration constructed as above. By Stokes' Theorem (note that, the directions of both $\Sigma$ and $H_r^+$ are future-directed), we obtain
$$
0=\int_{\Omega}d(e^{-f}w)=\int_{\partial \Omega}e^{-f}w=\int_{\widetilde{\Sigma}}e^{-f}w- \int_{\widetilde{H_r^+}}e^{-f}w+\int_{\widetilde{\partial\cal C}}e^{-f}w,$$
or
$$\int_{\widetilde{H_r^+}}e^{-f}w=\int_{\widetilde{\Sigma}}e^{-f}w+\int_{\widetilde{\partial\cal C}}e^{-f}w.$$

By Lemma \ref{lemw}, $\int_{\widetilde{H_r^+}}e^{-f}w\ge \Vol_f({\widetilde{H_r^+}})$ and $\int_{\widetilde{\Sigma}}e^{-f}w=\Vol_f(\widetilde{\Sigma}).$ Thus, we obtain the following estimate
$$\Vol_f(\widetilde{H_r^+})\le \Vol_f(\widetilde{\Sigma})+\int_{\widetilde{\partial\cal C}}e^{-f}w.$$
Since $\displaystyle\lim_{R\rightarrow\infty}\Vol_f(\widetilde{H_r^+})=\Vol_f(H_r^+)$ and $\displaystyle\lim_{R\rightarrow\infty}\Vol_f(\widetilde{\Sigma_R})=\Vol_f(\Sigma),$ so the remain we have to prove is $\displaystyle\lim_{R\rightarrow\infty}\int_{\widetilde{\partial\cal C}}e^{-f}w=0.$

Let $\overline{\partial{\cal C}}=S_R^{n-1}\times \left[-~\sqrt{R^2+r^2}, \sqrt{R^2+r^2}\;\right]$ be the part of $\partial{\cal C}$ between two parallel hyperplanes $t=\pm\sqrt{R^2+r^2}.$  Because $\widetilde{\partial{\cal C}}\subset \overline{\partial\cal C}$ and the density on $\overline{\partial{\cal C}}$ is  $e^{c-\frac{R^2}2}$ (a constant), we get 
$$\int_{\widetilde{\partial\cal C}}e^{-f}w\le  e^{c-\frac{R^2}2}\int_{\overline{\partial\cal C}}w.$$

Note that a unit normal vector field of ${\cal C}$ is of the form $(\textbf{x}, 0),$ where $|\textbf x|=1.$ Therefore 
$$\int_{\overline{\partial\cal C}}w=\int_{\overline{\partial{\cal C}}}\frac{\langle\nabla u, \textbf{x}\rangle}{\sqrt{1-|\nabla u|^2}}.$$
By the assumption $|\nabla u|$ is bounded away form 1,  we have
$$\int_{\overline{\partial{\cal C}}}\frac{\langle\nabla u, \textbf{x}\rangle}{\sqrt{1-|\nabla u|^2}}\le K\Vol(\overline{\partial{\cal C}})= 2K\sqrt{R^2+r^2}\Vol(S^{n-1}_R),$$
 where $K>0$ is a constant.
Finally, we get the following inequality
$$\int_{\widetilde{\partial\cal C}}e^{-f}w\le  2e^{c-\frac{R^2}2}K\sqrt{R^2+r^2}\Vol(S^{n-1}_R).$$
It is easy to see that, the right-hand side of the inequality goes to zero when $R$ approaches infinity and therefore the proof of the lemma is complete.
 \end{proof}

%====================================
As a consequence, we obtain the following Calabi-Bernstein type theorem.
\begin{thm}\label{main}
In $\Bbb G^n\times\Bbb R_1,$ the graph  $\Sigma$ of a function $u(\text{\bf x})=t$ over $\Bbb G^n,$ with  $|\nabla u|$ is bounded away from 1,  is $f$-maximal if and only if $u$ is constant,  i.e., $\Sigma$ is a spacelike hyperplane.
\end{thm}
\begin{proof}
It is clear that, if $u$ is constant, then $\Sigma$ is $f$-maximal.

Note that $H_r^+$ is the graph of the function $g(\textbf{x})=\sqrt{{\textbf x}^2+r^2},$ so
$$\Vol_f (H_r^+)=\int_{\Bbb R^n}e^{-f}\sqrt{\frac{r^2}{\textbf{x}^2+r^2}}dV_{_{\Bbb R^n}}.$$
It is easy to see that
$$\lim_{r\rightarrow\infty}\Vol_f (H_r^+)= 1.$$
Since
$$\begin{aligned}1&=\Vol_f \Bbb G^n=\int_{\Bbb R^n}e^{-f}dV_{_{\Bbb R^n}}\\
&\ge \int_{\Bbb R^n}e^{-f}\sqrt{1-|\nabla u|^2}dV_{_{\Bbb R^n}}=\Vol_f (\Sigma)\\
&\ge \lim_{r\rightarrow\infty}\Vol_f (H_r^+)=1,
\end{aligned}$$
we conclude that $|\nabla u|^2=0,$ i.e., $u$ is constant.
\end{proof}

%===============================================
{\bf Acknowledgements.} The authors would like to thank the  referee for carefully reading our paper and for his/her valuable comments and suggestions which helped to improve the manuscript. 
 This research is funded by Vietnam National Foundation for Science and Technology Development (NAFOSTED) under grant number 101.04.2014.26.

%===============================

\end{document}